\documentclass[11pt]{article}
\usepackage{amsmath,amsfonts,amssymb,amsthm} 
\usepackage[explicit]{titlesec}
\usepackage{graphicx}
\graphicspath{{Pictures/}} 
\usepackage{enumitem}
\setlist{nolistsep}
\usepackage{stmaryrd}
\usepackage{mathpple}
\usepackage{enumerate}
\usepackage{textcomp}
\usepackage[utf8]{inputenc}
\usepackage{mathtools}
\usepackage{epsfig}
\usepackage{xcolor}
\usepackage{array}
\usepackage{caption}
\usepackage{calc}
\usepackage{fancyhdr}
\usepackage[colorlinks=false]{hyperref}
\usepackage[noabbrev]{cleveref}
\usepackage{dsfont}
\usepackage{wrapfig}
\usepackage{fancyhdr}
\usepackage{float}
\usepackage[font={small}]{caption}
\usepackage{eurosym}
\usepackage{graphicx}
\usepackage{csquotes}
\usepackage[left=3cm,right=3cm,top=2cm,bottom=2.5cm]{geometry}

\DeclareMathAlphabet{\mathcal}{OMS}{cmsy}{m}{n}
\newcommand{\R}{\mathbb{R}}
\newcommand{\IP}{\mathbb{P}}

\newcommand{\N}{\mathbb{N}}

\newcommand{\E}{\mathbb{E}}

\newcommand{\IL}{\mathbb{L}}

\newcommand{\cB}{\mathcal{B}}

\newcommand{\F}{\mathcal{F}}

\newcommand{\cN}{\mathcal{N}}

\newcommand{\cL}{\mathcal{L}}

\newcommand{\1}{\mathds{1}}

\newcommand{\eps}{\varepsilon}

\def\build#1_#2^#3{\mathrel{\mathop{\kern 0pt#1}\limits_{#2}^{#3}}}

\numberwithin{equation}{section}

\setitemize[1]{label=--,partopsep=\parskip,topsep=-\parskip,itemsep=0pt,parsep=0pt}
\titleformat{\chapter}{\fontfamily{phv}\selectfont\LARGE\bfseries}{\thesection}{1em}{\fontfamily{phv}\selectfont\LARGE\bfseries #1}
\titleformat{\section}{\centering\fontfamily{phv}\selectfont\bfseries}{\thesection}{1em}{\centering\fontfamily{phv}\selectfont\bfseries #1}
\titleformat{\subsection}{\fontfamily{phv}\selectfont\small\bfseries\centering}{\thesubsection}{1em}{\fontfamily{phv}\selectfont\small\bfseries #1}
\titleformat{\subsubsection}{\fontfamily{phv}\selectfont\footnotesize\bfseries\centering}{\thesubsubsection}{1em}{\fontfamily{phv}\footnotesize\selectfont\bfseries #1}

\renewcommand*\thesection{\arabic{section}}

\makeatletter
\renewcommand{\@seccntformat}[1]{\llap{{\csname the#1\endcsname}\hspace{1em}}}   

\newtheoremstyle{thmit}{10pt}{10pt}
{\normalfont\itshape}
{}
{\small\bf\fontfamily{phv}\selectfont}
{.\;}
{0.25em}
{\small\fontfamily{phv}\selectfont\thmname{#1}\nobreakspace\footnotesize\thmnumber{#2}
\thmnote{\nobreakspace\the\thm@notefont\fontfamily{phv}\selectfont\footnotesize\bfseries\--\nobreakspace#3}}

\newtheoremstyle{rmq}
{5pt}
{5pt}
{\normalfont}
{}
{\small\bf\fontfamily{phv}\selectfont}
{\;}
{0.25em}
{\small\fontfamily{phv}\selectfont\thmname{#1}\nobreakspace\footnotesize\thmnumber{\@ifnotempty{#1}{}\@upn{#2}}
\thmnote{\nobreakspace\the\thm@notefont\fontfamily{phv}\selectfont\footnotesize\bfseries--\nobreakspace#3.}}
\makeatother

\newcounter{count}
\numberwithin{count}{section}
\newcounter{alpha}

\theoremstyle{thmit}
\newtheorem{theorem}[count]{Theorem\small}
\newtheorem{proposition}[alpha]{Proposition\small}

\theoremstyle{rmq}
\newtheorem{remark}[count]{Remark\small}

\renewenvironment{proof}[2]
{\paragraph{\fontfamily{phv}\selectfont\bfseries\small Proof of #1 \footnotesize #2.}}%
{\begin{flushright}
\qed
\end{flushright}}

\begin{document}
\pagestyle{fancy}

\title{\vspace{-3ex}\fontfamily{phv}\selectfont\bfseries\Large%
A martingale approach for Pólya urn processes}
\author{\fontfamily{phv}\selectfont\bfseries
Lucile Laulin}
\date{}
\AtEndDocument{\bigskip{\footnotesize%
  \textsc{Université de Bordeaux, Institut de Mathématiques de Bordeaux,
 UMR 5251, 351 Cours de la Libération, 33405 Talence cedex, France.} \par  
  \textit{E-mail adress :} \href{mailto:lucile.laulin@math.u-bordeaux.fr}{\texttt{lucile.laulin@math.u-bordeaux.fr}} \par
	}}

\maketitle

\centerline{
\begin{minipage}[c]{0.7\textwidth}
{\small\section*{Abstract}\vspace{-1ex}
This paper is devoted to a direct martingale approach for Pólya urn models asymptotic behaviour. A Pólya process is said to be small when the ratio of its remplacement matrix eigenvalues is less than or equal to $1/2$, otherwise it is called large. 
We find again some well-known results on the asymptotic behaviour for small and large urns processes. We also provide new almost sure properties for small urns processes.
}\end{minipage}
}

\setlength{\parindent}{0pt}

\section{Introduction}
At the inital time $n=0$, an urn is filled with $\alpha \geq 0$ red balls and $\beta \geq 0$ white balls. 
Then, at any time $n\geq 1$  one ball is drawn randomly from the urn and its color observed. If it is red it is then returned to the urn together with $a$ additional red balls and $b\geq 0$ white ones. If it is white it is then returned to the urn together with $c\geq 0$ additional red balls and $d$ white ones. The model corresponding replacement matrix is given, for $a,b,c,d\in\N$, by
\begin{equation}
	R= \begin{pmatrix}a & b \\ c & d\end{pmatrix}.
\end{equation}
The urn processe is said to be {\it balanced} if the total number of balls added at each step is a constant, $S =a+b=c+d \geq 1$.
Thanks to the balance assumption, $S$ is the maximum eigenvalue of $R^T$. Moreover, the second eigenvalue of $R^T$ is given by $m=a-c=d-b$. Throughout the rest of this paper, we shall denote 
\begin{equation*}
	\sigma = m/S\leq 1
\end{equation*}
the ratio of the two eigenvalues. It is straightforward that the respective eigenvectors of $R^T$ are given by
\begin{equation*}
	v_1 = \frac{S}{b+c} \begin{pmatrix} c \\ b\end{pmatrix} \hspace{1cm}\text{and}\hspace{1cm} v_2 = \frac{S}{b+c} \begin{pmatrix} 1 \\ -1\end{pmatrix}.
\end{equation*}
We can rewrite $R^T$ under the following form
	\begin{equation*}
		R^T = PDP^{-1} = 	\frac{1}{b+c} \begin{pmatrix} c & 1 \\ b & -1 \end{pmatrix}
	\begin{pmatrix} S & 0 \\ 0 & m \end{pmatrix}\begin{pmatrix} 1 & 1 \\ b & -c \end{pmatrix}.
	\end{equation*}
Hereafter, let us define the process $(U_n)$, the composition of the urn at time $n$, by
\begin{equation*}
	U_n=\begin{pmatrix}X_n \\ Y_n\end{pmatrix} \hspace{1cm}\text{and}\hspace{1cm} U_0=\begin{pmatrix}\alpha \\ \beta\end{pmatrix}
\end{equation*} where $X_n$ is the number of red balls and $Y_n$ is the number of white ones. Then, let $\tau=\alpha+\beta \geq 1$ and $\tau_n=\tau+nS$ be the number of ball inside the urn at time $n$. In particular, one can observe that $X_n + Y_n = \tau_n$ is a deterministic quantity.

The traditionnal Pólya urn model corresponds to the case where the replacement matrix $R$ is diagonal, while the generalized Pólya urn model corresponds to the case where the replacement matrix $R$ is at least triangular.

The questions about the asymptotic behavior of $(U_n)$ have been extensively studied, firstly by Freedman \cite{freedman65} and by many after, see for example \cite{Chauvin2011,Flajolet06,Flajolet05,Janson04,Pouyanne08,janson18}. We also refer the reader to Pouyanne's CIMPA summer school lectures 2014 \cite{cimpa14} for a very comprehensive survey on Pólya urn processes that has been a great source of inspiration. The reader may notice that this paper is related to Bercu \cite{Bercu18} on the elephant random walk. This is due to the paper of Baur and Bertoin \cite{Baur16} on connection between elephant random walks and Pólya-type urns.

Our strategy is to use the martingale theory \cite{Duflo97,Hall80} in order to propose a direct proof of the asymptotic normality associated with $(U_n)$. We also establish new refinements on the alm{}ost sure convergence of $(U_n)$.  The paper is organized as follows. In Section 2, we briefly present the traditional Pólya urn model, as well as the martingale related to this case. We establish the almost sure convergence and the asymptotic normality for this martingale. In Section 3, we present the generalized Pólya urn model with again the martingale related to this case, and we also give the main results for this model.  Hence, we first investigate small urn regime where $\sigma \leq 1/2$ and we establish the almost sure convergence, the law of iterated logarithm and the quadratic strong law for $(U_n)$. The asymptotic normality of the urn composition is also provided. We finally study the large urn where $\sigma >1/2$ and we prove the almost sure convergence as well as the mean square convergence of $(U_n)$ to a non-degenerate random vector whose moments are given. The proofs are postponed to Sections 4 and 5.

\section{Traditional Pólya urn model}

This model corresponds to the case where the replacement matrix is diagonal
\begin{equation*}
	R= \begin{pmatrix} S & 0 \\ 0 & S \end{pmatrix}.
\end{equation*}
It means that at any time $n\geq 1$,  one ball is drawn randomly from the urn, its color observed and it is then returned to the urn together with $S \geq 1$ additional balls of the same color.
Let us define the process $(M_n)$ by 
\begin{equation*}
	M_n= \frac{X_n}{\tau_n}
\end{equation*} 
and write 
\begin{equation}
 	X_n=\alpha+S\sum_{k=1}^n \eps_k
 \end{equation}
 where the conditional distribution of $\eps_{n+1}$ given the past up to time $n$ is $\cL(\eps_{n+1} |\F_n)=\cB(M_n)$. 
We clearly have 
	\begin{equation*}
		\E[M_{n+1}|\F_n] = M_n
	\end{equation*}
which means that $(M_n)$ is a martingale.
We have $\Delta M_{n+1}=\frac{S}{\tau_{n+1}}\big(\eps_{n+1}-M_n\big)$. Hence,
\begin{equation*}
	\E\bigl[\Delta M_{n+1}^2 | \F_n \bigr] = \frac{S^2}{\tau^2_{n+1}}\Big(\E\bigl[\eps_{n+1}^2 | \F_n \bigr] -M_n^2\Big)= \frac{S^2M_n(1-M_n)}{\tau_{n+1}^2}.
\end{equation*}
We now focus our attention on the asymptotic behavior of $(M_n)$.
\begin{theorem}
\label{T-tradi-as}
	The process $(M_n)$ converges to a random variable $M_\infty$ almost surely and in any $\IL^p$ for $p\geq1$. 
	The limit $M_\infty$ has a beta distribution, with parameters $\frac{\alpha}{S}$ and $\frac{\beta}{S}$.
\end{theorem}
\begin{remark}
	This results was first proved by Freedman, Theorem 2.2 in \cite{freedman65}.
\end{remark}

Our first new result on the gaussian fluctuation of $(M_n)$ is as follows.
\begin{theorem}
\label{T-tradi-dis}
We have the following convergence in distribution
\begin{equation}
\label{CVMNth-norm}	
  \sqrt{n}\frac{ M_\infty - M_n}{\sqrt{M_n(1-M_n)}} \underset{n\to\infty}{\overset{\cL}{\longrightarrow}} \cN\big(0,1\big)
\end{equation}
 \end{theorem}

\section{Gereralized Pólya urn model}

This model corresponds to the case where the replacement matrix is not diagonal, 
\begin{equation*}
	R= \begin{pmatrix} a & b \\ c & d \end{pmatrix}.
\end{equation*}
Let us rewrite 
\begin{equation*}
	X_n=\alpha + a\sum_{k=1}^n \eps_k + c\sum_{k=1}^n (1-\eps_k)
\end{equation*}
where the conditional distribution of $\eps_{n+1}$ given the past up to time $n$ is $\cL(\eps_{n+1} |\F_n)=\cB(\tau_n^{-1}X_n)$. We have
\begin{equation*}
	U_{n+1}= U_n +R^T \begin{pmatrix}	\eps_{n+1} \\ 1-\eps_{n+1}\end{pmatrix}
\end{equation*}
and
\begin{equation*}
	U_{n} - \E[U_n] = \begin{pmatrix} X_n -\E[X_n] \\ Y_n - \E[Y_n] \end{pmatrix}
				 = \big(X_n -\E[X_n]\big)\begin{pmatrix}	1 \\ -1 \end{pmatrix}
					=  \frac{b+c}{S} \big( X_n -\E[X_n] \big) v_2.
\end{equation*}
Hence, we obtain that
\begin{eqnarray}
	 \E\big[U_{n+1} - \E[U_{n+1}]|\F_n\big] & = & \nonumber U_{n} - \E[U_n] + R^T\E\Big[ \begin{pmatrix}	\eps_{n+1} \\ 1-\eps_{n+1}\end{pmatrix} - \E\big[ \begin{pmatrix}	
	 \eps_{n+1} \\ 1-\eps_{n+1}\end{pmatrix}\big]|\F_n\Big] \\ \nonumber
						& = & \big(I_2 + \tau_n^{-1}R^T\big)\Big(U_n - \E[U_n]\Big) \\ \nonumber
						& = & (X_n - E[X_n]\big)\big(I_2 + \tau_n^{-1}R^T\big)\begin{pmatrix}1  \\ -1\end{pmatrix} \\\nonumber
						& = & \big(1+\tau_n^{-1}m\big)\big(X_n - E[X_n]\big)\begin{pmatrix}1  \\ -1\end{pmatrix} \\ 
						& = & \big(1+\tau_n^{-1}m\big)\big(U_{n} - \E[U_n]\big). \label{EspUn-gene}
\end{eqnarray}
Finally, denote
\begin{equation}
\label{DEF-sigma-n}
	\sigma_n=\prod_{k=0}^{n-1}\big(1+\tau_k^{-1}m\big)^{-1}=\frac{\Gamma(n+\frac{\tau}{S})\Gamma(\frac{\tau}{S}+\sigma)}{\Gamma(\frac{\tau}{S})\Gamma(n+\frac{\tau}{S}+\sigma)}.
\end{equation}
One can observe that
\begin{equation}
\label{LIMIT-sigma-n}
	\lim_{n\to\infty} {n^\sigma}\sigma_n=\frac{\Gamma(\frac{\tau}{S}+\sigma)}{\Gamma(\frac{\tau}{S})}.
\end{equation}
Hereafter, we define the process $(M_n)$ by 
\begin{equation}
\label{defMn-gene}
	M_n = \sigma_n\big(U_n - \E[U_n]\big).
\end{equation}
Thanks to equation \eqref{EspUn-gene} we immediatly get that
	\begin{equation*}
	\label{Mn-martg}
		\E[M_{n+1}|\F_n] =  M_n.
	\end{equation*}
Hence, the sequence $(M_n)$ is a locally bounded and square integrable martingale.
We are now allowed to compute the quadratic variation of $(M_n)$. First of all
\begin{equation}
	\Delta M_{n+1}
	= m \sigma_{n+1} \big( \eps_{n+1} -\E[\eps_{n+1}|\F_n] \big) \begin{pmatrix} 1 \\ -1\end{pmatrix} 
	= m \sigma_{n+1} \big( \eps_{n+1} -\tau_n^{-1} X_n \big) \begin{pmatrix} 1 \\ -1\end{pmatrix} 
					\label{DeltaMn-1}.
\end{equation}
Moreover,
\begin{equation}
\label{ESP-eps-quad}
	\E\big[\big(\eps_{n+1} -\tau_n^{-1} X_n \big)^2\big | \F_n ] = \tau_n^{-1} X_n \big(1-\tau_n^{-1} X_n\big).
\end{equation}
 Consequently, we obtain from \eqref{DeltaMn-1} and \eqref{ESP-eps-quad} that
\begin{equation}
 	\E\big[\Delta M_{n+1} \Delta M_{n+1}^T \big | \F_n ] = 
		m^2\sigma_{n+1}^{2}\tau_n^{-1} X_n \big(1-\tau_n^{-1} X_n\big) \begin{pmatrix} 1 & -1 \\ -1 & 1 \end{pmatrix}.
	\label{ESP-delta-Mn-T}
 \end{equation} 
Therefore
\begin{eqnarray}
	\langle M\rangle_n & = & \sum_{k=0}^{n-1} \E\big[\Delta M_{k+1} \Delta M_{k+1}^T \big | \F_k ] \nonumber\\ 
						& = &	m^2 \begin{pmatrix}1 & -1 \\ -1 & 1\end{pmatrix}
						\sum_{k=0}^{n-1}\sigma_{k+1}^{2}\tau_k^{-1} X_k \big(1-\tau_k^{-1} X_k\big) \label{quadMn}.
\end{eqnarray}
It is not hard to see that
\begin{equation}
\label{Tr-wn}
 	\text{Tr} \langle M\rangle_n \leq m^2 w_n \hspace{1cm}\text{where}\hspace{1cm} w_n = \sum_{k=1}^{n}\sigma_{k}^{2}.
 \end{equation}
The asymptotic behavior of $(M_n)$ is closely related to the one of $(w_n)$ with the following trichotomy
\begin{itemize}
	\item The diffusive regime where $\sigma <1/2$ : the urn is said to be small and we have
	\begin{equation*}
		\lim_{n\to\infty} \frac{w_n}{n^{1-2\sigma}} = \frac{\lambda^2}{1-2\sigma} \hspace{1cm} \text{where} 
	\hspace{1cm} \lambda= \frac{\Gamma(\frac{\tau}{S}+\sigma)}{\Gamma(\frac{\tau}{S})}.
	\end{equation*}
	\item The critical regime where $\sigma = 1/2$ : the urn is said to be critically small and we have
	\begin{equation*}
		\lim_{n\to\infty} \frac{w_n}{\log n} =  \frac{\Gamma(\frac{\tau}{S}+\frac{1}{2})}{\Gamma(\frac{\tau}{S})}.
	\end{equation*}
	\item The superdiffusive regime where $\sigma > 1/2$ : the urn is said to be large and we have
	\begin{equation*}
		\lim_{n\to\infty} {w_n} = \sum_{k=0}^{\infty} 
	\Big(\frac{\Gamma(k+\frac{\tau}{S})\Gamma(\frac{\tau}{S}+\sigma)}{\Gamma(\frac{\tau}{S})\Gamma(k+\frac{\tau}{S}+\sigma)}\Big)^2.
	\end{equation*}
\end{itemize}

\begin{proposition}
\label{P-ESP-limit}
We have for small and large urns
	\begin{equation}
		\E[U_n] = n v_1 + \sigma_n^{-1} \Big(\frac{b\alpha - c\beta}{S}\Big)v_2 + \frac{\tau}{S}v_1. \label{ESP-sigma}
	\end{equation}
\end{proposition}

\begin{proof}{Proposition}{\ref{P-ESP-limit}}
	First of all, denote $\Lambda_n = I_2 + \tau_n^{-1}R^T = P \big(I_2 +\tau_n^{-1}D\big) P^{-1}$ and $T_n= \prod_{k=0}^{n-1} \Lambda_k$. For any $n\in\N$, $T_n$ is diagonalisable and 
	\begin{equation*}
		T_n = P D_n P^{-1} = \frac{1}{b+c}\begin{pmatrix} c & 1 \\ b & -1\end{pmatrix}
			\begin{pmatrix} \tau_n/\tau & 0 \\ 0 & \sigma_n^{-1}\end{pmatrix}
			\begin{pmatrix} 1 & 1 \\ b & -c\end{pmatrix}.
	\end{equation*}
	Since $E[U_{n+1}|\F_n]= \Lambda_n U_n$ we easily get that $\E[U_n]= T_n U_0$, which leads to
	\begin{eqnarray*}
		\E[U_n] & = & \frac{1}{b+c}\Big(\frac{\tau_n}{\tau} \begin{pmatrix} c & c \\ b & b \end{pmatrix} 
				+ \sigma_n^{-1} \begin{pmatrix} b & -c \\ -b & c \end{pmatrix}\Big)U_0\\
				& = & n v_1 + \frac{\tau}{S}v_1 + \sigma_n^{-1} \frac{b\alpha - c\beta}{S}v_2.
	\end{eqnarray*}
\end{proof}

\subsection{Small urns}

The almost sure convergence of $(U_n)$ for small urns is due to Janson, Theorem 3.16 in \cite{Janson04}.
\begin{theorem}
\label{T-general-as-small}
	When the urn is small, $\sigma <1/2$, we have the following convergence
	\begin{equation}
	\label{general-as-small}
		\lim_{n\to\infty}\frac{U_n}{n} = v_1
	\end{equation}
	almost surely and in any $\IL^p$, $p\geq 1$.	
\end{theorem}

Our new refinements on the almost sure rates of convergence are as follows.

\begin{theorem}
\label{T-general-LFQLIL-small}
When the urn is small and $bc\neq0$, we have the quadratic strong law
\begin{equation}
\label{LFQsmall}
 \lim_{n \rightarrow \infty} \frac{1}{\log n} \sum_{k=1}^n \frac{1}{k^2}(U_k-kv_1) (U_k-kv_1)^T=\frac{1}{1-2\sigma}\frac{bcm^2}{(b+c)^2}\begin{pmatrix}1 & -1 \\ -1 & 1\end{pmatrix} \hspace{1cm} \text{a.s.}
\end{equation}
In particular,
\begin{equation}
\lim_{n \rightarrow \infty} \frac{1}{ \log n} \sum_{k=1}^n  \frac{\|U_k-kv_1\|^2}{k^2}=  \frac{2}{1-2\sigma}\frac{bcm^2}{(b+c)^2} \hspace{1cm} \text{a.s.}
\label{LFQNORMsmall}
\end{equation}
Moreover, we have the law of iterated logarithm
\begin{equation}
 \limsup_{n \rightarrow \infty} \frac{\|U_n-nv_1\|^2}{2 n \log \log n} =  \frac{2}{1-2\sigma}\frac{bcm^2}{(b+c)^2} \hspace{1cm} \text{a.s.}  
\label{LILsmall}
\end{equation}
\end{theorem}
\begin{remark}
	The law of iterated logarithm for $(X_n)$ was previously established by Bai, Hu and Zhang via a strong approximation argument, see Corollary 2.1 in \cite{bai02}.
\end{remark}
\begin{theorem}
\label{T-general-dis-small}
	When the urn is small and $bc\neq0$, we have the following convergence asymptotic normality
	\begin{equation}
		\frac{U_n - n v_1}{\sqrt{n}} \overset{\cL}{\underset{n\to\infty}{\longrightarrow}} 
	\cN \big(0,\Gamma\big)
	\end{equation}
	where $\displaystyle \Gamma =\frac{1}{1-2\sigma}\frac{bcm^2}{(b+c)^2}\begin{pmatrix}1 & -1 \\ -1 & 1\end{pmatrix}$.
\end{theorem}
\begin{remark}
	An invariance principle for $(X_n)$ was proved by Gouet, see Proposition 2.1 in \cite{gouet93}.
\end{remark}

\subsection{Critically small urns}

The almost sure convergence of $(U_n)$ for critically small urns is again due to Janson, Theorem 3.16 in \cite{Janson04}.
\begin{theorem}
\label{T-general-as-crit}
	When the urn is critically small, $\sigma=1/2$, we have the following convergence
	\begin{equation}
	\label{general-as-crit}
		\lim_{n\to\infty}\frac{U_n}{n} = v_1
	\end{equation}
	almost surely and in any $\IL^p$, $p\geq 1$.	
\end{theorem}

Once again, we have some refinements on the almost sure rates of convergence.

\begin{theorem}
\label{T-general-LFQLIL-crit}
When the urn is critically small and $bc\neq0$, we have the quadratic strong law
\begin{equation}
\label{LFQsmallcrit}
 \lim_{n \rightarrow \infty} \frac{1}{\log \log n} \sum_{k=1}^n \frac{1}{(k\log k)^2}(U_k-kv_1) (U_k-kv_1)^T=bc\begin{pmatrix}1 & -1 \\ -1 & 1\end{pmatrix} \hspace{1cm} \text{a.s.}
\end{equation}
In particular,
\begin{equation}
\lim_{n \rightarrow \infty} \frac{1}{ \log \log n} \sum_{k=1}^n  \frac{\|U_k-kv_1\|^2}{(k\log k)^2}= 2 bc \hspace{1cm} \text{a.s.}
\label{LFQNORMsmallcrit}
\end{equation}
Moreover, we have the law of iterated logarithm
\begin{equation}
 \limsup_{n \rightarrow \infty} \frac{\|U_n-nv_1\|^2}{2 \log n \log \log \log n} =  2 bc \hspace{1cm} \text{a.s.}
\label{LILsmallcrit}
\end{equation}
\end{theorem}
\begin{remark}
	The law of iterated logarithm for $(X_n)$ 
	was also established by Bai, Hu and Zhang via a strong approximation argument, see Corollary 2.2 in \cite{bai02}.
\end{remark}
\begin{theorem}
\label{T-general-dis-crit}
	When the urn is critically small and $bc\neq0$, we have the following asymptotic normality
	\begin{equation}
		\frac{U_n - n v_1}{\sqrt{n\log n}} \overset{\cL}{\underset{n\to\infty}{\longrightarrow}} 
	\cN \big(0,\Gamma\big)
	\end{equation}
	where $\displaystyle \Gamma =bc\begin{pmatrix}1 & -1 \\ -1 & 1\end{pmatrix}$.
\end{theorem}
\begin{remark}
	An invariance principle for $(X_n)$ was also proven by Gouet, see Proposition 2.1 in \cite{gouet93}.
\end{remark}
\subsection{Large urns}

The convergences of $n^{-\sigma}(U_n-nv_1)$ to $Wv_2$ first appeared in Pouyanne \cite{Pouyanne08}, Theorem 3.5. The almost sure convergence of $(U_n)$ for large urns is again due to Janson, Theorem 3.16 in \cite{Janson04}. The explicit calculation of the moments of $W$ are new.
\begin{theorem}
\label{T-general-as-large}
	When the urn is large, $\sigma>1/2$, we have the following convergence
	\begin{equation}
	\label{EQ-as-large}
		\lim_{n\to\infty}\frac{U_n}{n} = v_1
	\end{equation}
	almost surely and in any $\IL^p$, $p\geq 1$. Moreover, we also have
	\begin{equation}
	\label{CVGLW}
		\lim_{n\to\infty}\frac{U_n-n v_1}{n^\sigma} = W v_2
	\end{equation}
	almost surely and in $\IL^2$, where $W$ is a real-valued random variable and
	\begin{equation}
		\E[W]= \frac{\Gamma(\frac{\tau}{S})}{\Gamma(\frac{\tau}{S}+\sigma)}\frac{b\alpha - c\beta}{S},
	\end{equation}
	\begin{equation}
		\E[W^2] = \sigma^2\frac{\Gamma(\frac{\tau}{S})}{\Gamma(\frac{\tau}{S}+2\sigma)}
			\Big(\frac{bc}{2\sigma-1}\frac{\tau}{S} 
			+
			(b-c)\frac{b\alpha  - c\beta}{\sigma S}
			+ \frac{(b\alpha  - c\beta)^2}{\sigma^2S^2} \Big).
	\end{equation}
\end{theorem}

\section{Proofs of the almost sure convergence results}

\subsection{Generalized urn model -- small urns}

\begin{proof}{Theorem}{\ref{T-general-as-small}}
	We denote the maximum eigenvalue of $\langle M\rangle_n$ by $\lambda_{max} \langle M\rangle_n$. We make use of the strong law of large numbers 
	for martingales given e.g. by Theorem 4.3.15 of \cite{Duflo97}, that is for any $\gamma >0$,
	\begin{equation*}
		\frac{\|M_n\|^2}{\lambda_{max} \langle M\rangle_n}  = o \big((\log \text{Tr} \langle M\rangle_n)^{1+\gamma}\big) \hspace{1cm}\text{a.s.}
	\end{equation*}
	 It follows from \eqref{Tr-wn} that
	\begin{equation*}
		\|M_n\|^2  = o \big( w_n (\log w_n)^{1+\gamma}\big) \hspace{1cm}\text{a.s.}
	\end{equation*}
	which implies
	\begin{equation*}
		\|M_n\|^2  = o \big( n^{1-2\sigma} (\log n)^{1+\gamma}\big) \hspace{1cm}\text{a.s.}
	\end{equation*}
	Hence, we deduce from \eqref{LIMIT-sigma-n} and \eqref{defMn-gene} that
	\begin{equation*}
		\|U_n-\E[U_n]\|^2 = o \big( n (\log n)^{1+\gamma}\big) \hspace{1cm}\text{a.s.}
	\end{equation*}
	which completes the proof for the almost sure convergence. The convergence in any $\IL^p$ for $p\geq 1$ holds since $n^{-1}\|U_n-\E[U_n]\|$ is uniformly bounded by $2\sqrt{2}(\tau +S)$.
\end{proof}

\begin{proof}{Theorem}{\ref{T-general-LFQLIL-small}}
	We shall make use of Theorem 3 of \cite{Bercu04}. For any $u\in\R^2$ let $M_n(u)= \langle u , M_n\rangle$ and denote
	$\displaystyle{f_n=\frac{\sigma_n^2}{w_n}}$. 
	We have from \eqref{LIMIT-sigma-n} that $f_n$ is equivalent to $(1-2\sigma)n^{-1}$ and converges to 0. Moreover, we obtain from equations \eqref{quadMn}, \eqref{general-as-small} and Toeplitz lemma that
		\begin{eqnarray*}
	\lim_{n\to\infty} \frac{1}{w_n}\langle M\rangle_n & = & 
		\lim_{n\to\infty} \frac{m^2}{w_n}\begin{pmatrix}1 & -1 \\ -1 & 1\end{pmatrix}\sum_{k=0}^{n-1}\sigma_{k+1}^{2}\tau_k^{-1} X_k \big(1-\tau_k^{-1} X_k\big) \\ 
		& = & \frac{bcm^2}{(b+c)^2}\begin{pmatrix}1 & -1 \\ -1 & 1\end{pmatrix}\hspace{1cm}\text{a.s.}
		\end{eqnarray*}
		which implies that
		\begin{equation}
		\label{Mnwnrates}
			\lim_{n\to\infty} \frac{1}{w_n}\langle M\rangle_n = (1-2\sigma)\Gamma \hspace{1cm}\text{a.s.}
		\end{equation}
	Therefore, we get from \eqref{Mnwnrates} that 
	\begin{equation*}
	\lim_{n\to\infty} \frac{1}{\log w_n} \sum_{k=1}^n f_k \Big(\frac{M_k(u)^2}{w_k}\Big) = (1-2\sigma)u^T \Gamma u\hspace{1cm}\text{a.s.}
	\end{equation*}
	which leads to
	\begin{equation*}
	\lim_{n\to\infty} \frac{1}{\log n} \sum_{k=1}^n f_k^2  u^T (U_k-E[U_k]) (U_n-E[U_k])^T u = (1-2\sigma)^2 u^T \Gamma u\hspace{1cm}\text{a.s.}
	\end{equation*}
	Furthermore, we have from \eqref{ESP-sigma} that $\E[U_n]$ is equivalent to $nv_1$. Consequently, we obtain that
	\begin{equation*}
	\lim_{n\to\infty} \frac{1}{\log n} \sum_{k=1}^n \frac{1}{k^2} (U_k-kv_1) (U_k-kv_1)^T =  \Gamma \hspace{1cm}\text{a.s.}
	\end{equation*}
We now focus our attention on the law of iterated logarithm. We already saw that
\begin{equation*}
 	\sum_{n=1}^{\infty} \frac{\sigma_n^4}{w_n^2}< \infty.
 \end{equation*} Hence, it follows from the law of iterated logarithm for real martingales that first appeared in Stout \cite{Stout70,Stout74}, that for any $u\in\R^d$,
\begin{eqnarray*}
\underset{n \to \infty}{\lim \sup} \frac{1}{\sqrt{2w_n\log\log w_n}}M_n(u) & = & - \underset{n \to \infty}{\lim \inf} \frac{1}{\sqrt{2w_n\log\log w_n}}M_n(u)  \\
								& = & \sqrt{(1-2\sigma) u^T \Gamma u }\hspace{1cm} \text{a.s.}
\end{eqnarray*}
Consequently, as $M_n(u) =\sigma_n\langle u, U_n - \E[U_n] \rangle$, we obtain that
\begin{eqnarray*}
\underset{n \to \infty}{\lim \sup} \frac{1}{\sqrt{2 n\log\log n}}\langle u, U_n - \E[U_n]  \rangle & = & - \underset{n \to \infty}{\lim \inf} \frac{1}{\sqrt{2n\log\log n}}\langle u, U_n - \E[U_n]  \rangle  \\
								& = & \sqrt{ u^T \Gamma u }\hspace{1cm} \text{a.s.}
\end{eqnarray*}
In particular, for any vector $u\in\R^2$ 
\begin{equation*}
\underset{n \to \infty}{\lim \sup} \frac{1}{2 n\log\log n} u^T (U_n - \E[U_n])(U_n - \E[U_n]) u = u^T \Gamma u \hspace{1cm} \text{a.s.} 
\end{equation*}
Finally, we deduce once again from \eqref{ESP-sigma}
\begin{equation*}
\underset{n \to \infty}{\lim \sup} \frac{1}{2 n\log\log n} (U_n - nv_1 ) (U_n - nv_1)^T= \Gamma \hspace{1cm} \text{a.s.} 
\end{equation*}
which completes the proof of Theorem \ref{T-general-LFQLIL-small}.
\end{proof}

\subsection{Generalized urn model -- critically small urns}

\begin{proof}{Theorem}{\ref{T-general-as-crit}}
	Again, we make use of the strong law of large numbers for martingales given e.g. by Theorem 4.3.15 of \cite{Duflo97}, that is for any $\gamma >0$,
	$$\frac{\|M_n\|^2}{\lambda_{max} \langle M\rangle_n}  = o \big((\log \text{Tr} \langle M\rangle_n)^{1+\gamma}\big) \hspace{1cm}\text{a.s.} $$
	 Since $\text{Tr} \langle M\rangle_n \leq m^2 w_n$ and the quadratic version of $M_n$ is a semi-definite positive matrix we have $\lambda_{max} \langle M\rangle_n \leq m^2 w_n$ so that
	$$\|M_n\|^2  = o \big( w_n (\log w_n)^{1+\gamma}\big) \hspace{1cm}\text{a.s.}$$
	which implies
	$$\|M_n\|^2  = o \big( \log n (\log \log n)^{1+\gamma}\big) \hspace{1cm}\text{a.s.}$$
	Moreover, by definition of $M_n$ and using $\sigma_n$ equivalent we get
	$$\|U_n-\E[U_n]\|^2 = o \big( \sqrt{n} \log n (\log \log n)^{1+\gamma}\big) \hspace{1cm}\text{a.s.} $$
	which completes the proof for the almost sure convergence. The convergence in any $\IL^p$ for $p\geq 1$ holds by the same arguments as in the proof of Theorem \ref{T-general-as-small}.
	\end{proof}

\begin{proof}{Theorem}{\ref{T-general-LFQLIL-crit}}
	We shall once again make use of Theorem 3 of \cite{Bercu04}. For any $u\in\R^2$ let $M_n(u)= \langle u , M_n\rangle$ and denote
	$\displaystyle{f_n=\frac{\sigma_n^2}{w_n}}$. 
	We have from \eqref{LIMIT-sigma-n} that $f_n$ is equivalent to $(n\log n){-1}$ and converges to 0. When $\sigma=1/2$ we have $b+c=m$. Moreover, we obtain from equations \eqref{quadMn}, \eqref{general-as-crit} and Toeplitz lemma that
		\begin{eqnarray*}
	\lim_{n\to\infty} \frac{1}{w_n}\langle M\rangle_n & = & 
		\lim_{n\to\infty} \frac{m^2}{w_n}\begin{pmatrix}1 & -1 \\ -1 & 1\end{pmatrix}\sum_{k=0}^{n-1}\sigma_{k+1}^{2}\tau_k^{-1} X_k \big(1-\tau_k^{-1} X_k\big) \\ 
		& = & {bc}\begin{pmatrix}1 & -1 \\ -1 & 1\end{pmatrix}\hspace{1cm}\text{a.s.}
		\end{eqnarray*}
		which implies that
		\begin{equation}
		\label{Mnwnrates-crit}
			\lim_{n\to\infty} \frac{1}{w_n}\langle M\rangle_n = \Gamma \hspace{1cm}\text{a.s.}
		\end{equation}
	Therefore, we get from \eqref{Mnwnrates} that 
	\begin{equation*}
	\lim_{n\to\infty} \frac{1}{\log w_n} \sum_{k=1}^n f_k \Big(\frac{M_k(u)^2}{w_k}\Big) = u^T \Gamma u\hspace{1cm}\text{a.s.}
	\end{equation*}
	which leads to
	\begin{equation*}
	\lim_{n\to\infty} \frac{1}{\log \log n} \sum_{k=1}^n f_k^2  u^T (U_k-E[U_k]) (U_n-E[U_k])^T u = u^T \Gamma u\hspace{1cm}\text{a.s.}
	\end{equation*}
	Consequently, we obtain from \eqref{ESP-sigma} that
	\begin{equation*}
	\lim_{n\to\infty} \frac{1}{\log \log n} \sum_{k=1}^n \frac{1}{(k\log k)^2} (U_k-kv_1) (U_k-kv_1)^T =  \Gamma \hspace{1cm}\text{a.s.}
	\end{equation*}
We now focus our attention on the law of iterated logarithm. It is not hard to see that
\begin{equation*}
 	\sum_{n=1}^{\infty} \frac{\sigma_n^4}{w_n^2}< \infty.
 \end{equation*} Hence, it follows from the law of iterated logarithm for real martingales that first appeared in Stout \cite{Stout70,Stout74}, that for any $u\in\R^d$,
\begin{eqnarray*}
\underset{n \to \infty}{\lim \sup} \frac{1}{\sqrt{2w_n\log\log w_n}}M_n(u) & = & - \underset{n \to \infty}{\lim \inf} \frac{1}{\sqrt{2w_n\log\log w_n}}M_n(u)  \\
								& = & \sqrt{u^T \Gamma u }\hspace{1cm} \text{a.s.}
\end{eqnarray*}
Consequently, we obtain that
\begin{eqnarray*}
\underset{n \to \infty}{\lim \sup} \frac{1}{\sqrt{2 \log n\log\log\log n}}\langle u, U_n - \E[U_n]  \rangle & = & - \underset{n \to \infty}{\lim \inf} \frac{1}{\sqrt{2\log n\log\log\log n}}\langle u, U_n - \E[U_n]  \rangle  \\
								& = & \sqrt{ u^T \Gamma u }\hspace{1cm} \text{a.s.}
\end{eqnarray*}
In particular, for any vector $u\in\R^2$ 
\begin{equation*}
\underset{n \to \infty}{\lim \sup} \frac{1}{2\log n\log\log\log n} u^T (U_n - \E[U_n])(U_n - \E[U_n]) u = u^T \Gamma u \hspace{1cm} \text{a.s.} 
\end{equation*}
Finally, we deduce once again from \eqref{ESP-sigma} that
\begin{equation*}
\underset{n \to \infty}{\lim \sup} \frac{1}{2\log n\log\log\log n} (U_n - nv_1 ) (U_n - nv_1)^T= \Gamma \hspace{1cm} \text{a.s.} 
\end{equation*}
which completes the proof of Theorem \ref{T-general-LFQLIL-crit}.
\end{proof}
 
\subsection{Generalized urn model -- large urns}

\begin{proof}{Theorem}{\ref{T-general-as-large}}
	First, as $\text{Tr}\langle M\rangle_n \leq m^2 w_n < \infty$, we have that $(M_n)$ converges almost surely to a random vector $Mv_2$, where $M$ is a real-valued random variable and 
	\begin{equation*}
	 	\lim_{n\to\infty} \sigma_n  \big( X_n -\E[X_n] \big) =\frac{S}{b+c} M = \frac{1}{1-\sigma} M  \hspace{1cm}\text{a.s.}
	 \end{equation*} 
	Hence, it follows from \eqref{defMn-gene} that
	\begin{equation}
	\label{CV-Mv2}
		\lim_{n\to\infty} \sigma_n(U_n - \E[U_n])= Mv_2 \hspace{1cm}\text{a.s.}
	\end{equation}
	which implies via \eqref{LIMIT-sigma-n} that
	\begin{equation*}
		\lim_{n\to\infty} \sigma_n(U_n - \E[U_n])=\lim_{n\to\infty} \frac{\lambda }{n^\sigma}\| U_n - \E[U_n]\| = \|Mv_2\|\hspace{1cm}\text{a.s.}
	\end{equation*}
	Therefore, we obtain that 
	\begin{equation}
	\label{Un-large-cv}
		\lim_{n\to\infty} \frac{\| U_n - \E[U_n]\|}{n} = 0 \hspace{1cm}\text{a.s.}
	\end{equation}
	Hence, we deduce \eqref{EQ-as-large} from \eqref{CV-Mv2} and \eqref{Un-large-cv}.
	The convergence in any $\IL^p$ for $p\geq 1$ holds again by the same arguments as before.
	We now focus our attention on equation \eqref{CVGLW}. We have from \eqref{ESP-sigma} and \eqref{CV-Mv2} that
	\begin{equation*}
		\lim_{n\to \infty } \sigma_n \big( U_n -\E[U_n] \big)  
		= \lim_{n\to \infty }\sigma_n \big(U_n -n v_1\big) - \Big(\frac{b\alpha - c\beta}{S}\Big)v_2 = Mv_2 \hspace{1cm}\text{a.s.} 
	\end{equation*}
	Consequently,
	\begin{equation*}
	\label{lim-UnW}
	\lim_{n\to \infty }\frac{U_n -n v_1}{n^\sigma}  = Wv_2 \hspace{1cm}\text{a.s.}
	\end{equation*} 
	where the random variable W is given by
	\begin{equation}
	\label{defW}	 
	W = \frac{1}{\lambda}\big(M + \frac{b\alpha - c\beta}{S}\big)						
	\end{equation}
	Hereafter, as 
	\begin{equation*}
		\E\big[\|M_n\|^2\big] = \E\big[\text{Tr}\langle M\rangle_n] \leq m^2 w_n,
	\end{equation*}
	we get that
	\begin{equation*}
		\sup_{n\geq 1} \E\big[\|M_n\|^2\big] < \infty
	\end{equation*}
	which means that $(M_n)$ is a martingale bounded in $\IL^2$, thus converging in $\IL^2$.
	Finally, as $\E[M_n]=0$ and $(M_n)$ converges in $\IL^1$ to $M$, $\E[M]=0$. Hence, we find from \eqref{lim-UnW} that
	\begin{equation*}
		\E[W] = \frac{\Gamma(\frac{\tau}{S})}{\Gamma(\frac{\tau}{S}+\sigma)}\frac{b\alpha - c\beta}{S}.
	\end{equation*}
	We shall now proceed to the computation of $\E[W^2]$. We have from \eqref{defW} that
	\begin{equation}
	\label{ESP-MW-2}
	 	\E[M^2] = {\lambda^2}\E[W^2] - \frac{(b\alpha - c\beta)^2}{S^2},
	 \end{equation} so that we only need to find $\E[M^2]$. It is not hard to see that 
	\begin{equation*}
		\E\big[(X_{n+1}-\E[X_{n+1}])^2\big] = (1+2m\tau_n^{-1})\E\big[(X_{n}-\E[X_{n}])^2\big] + m^2 \tau_n^{-1}\E[X_n]\big(1-\tau_n^{-1}\E[X_n])
	\end{equation*}
	wich leads to
	\begin{eqnarray*}
	\label{ESP-2-Xn}
		\E\big[X_n-\E[X_n]\big]^2
		& = & m^2\frac{\Gamma(n+\frac{\tau}{S}+2\sigma)}{\Gamma(n+\frac{\tau}{S})}
		\sum_{k=0}^{n-1} 
		\frac{\Gamma(k+1+\frac{\tau}{S})}{\Gamma(k+1+\frac{\tau}{S}+2\sigma)} 
		\tau_k^{-1}\E[X_k]\big(1-\tau_k^{-1}\E[X_k]) \\
		& = & \frac{\sigma^2}{(1-\sigma)^2}\frac{\Gamma(n+\frac{\tau}{S}+2\sigma)}{\Gamma(n+\frac{\tau}{S})} S_n.
	\end{eqnarray*}
	It follows from \eqref{ESP-sigma} that
	\begin{eqnarray*}
		S_n & = & {(b+c)^2}\sum_{k=0}^{n-1} \tau_k^{-1}\E[X_k]\big(1-\tau_k^{-1}\E[X_k])
		\frac{\Gamma(k+1+\frac{\tau}{S})}{\Gamma(k+1+\frac{\tau}{S}+2\sigma)} \\
			& = & 
			bc A_n +(b-c)\frac{b\alpha  - c\beta}{S}\frac{\Gamma(\frac{\tau}{S})}{\Gamma(\frac{\tau}{S}+\sigma)}B_n - \frac{(b\alpha  - c\beta)^2}{S^2} \frac{\Gamma(\frac{\tau}{S})^2}{\Gamma(\frac{\tau}{S}+\sigma)^2} C_n
	\end{eqnarray*}
	where $A_n$, $B_n$ and $C_n$ are as follows, and we obtain from lemma B.1 in \cite{Bercu18} that 
	\begin{equation*}
		A_n = \sum_{k=1}^{n}\frac{\Gamma(k+\frac{\tau}{S})}{\Gamma(k+\frac{\tau}{S}+2\sigma)} =\frac{1}{2\sigma-1}
		 	\big(\frac{\Gamma(\frac{\tau}{S}+1)}{\Gamma(\frac{\tau}{S}+2\sigma)} - \frac{\Gamma(n+\frac{\tau}{S}+1)}{\Gamma(n+\frac{\tau}{S}+2\sigma)}\big),
	\end{equation*}
	\begin{equation*}
		B_n =  \sum_{k=1}^{n}\frac{\Gamma(k-1+\frac{\tau}{S}+\sigma)}{\Gamma(k+\frac{\tau}{S}+2\sigma)}
		  = \frac{1}{\sigma}
		 	\big(\frac{\Gamma(\frac{\tau}{S}+\sigma)}{\Gamma(\frac{\tau}{S}+2\sigma)} - \frac{\Gamma(n+\frac{\tau}{S}+\sigma)}{\Gamma(n+\frac{\tau}{S}+2\sigma)}\big),
	\end{equation*}	
	\begin{equation*}
		C_n = \sum_{k=1}^{n}\frac{\Gamma(k-1+\frac{\tau}{S}+\sigma)^2}{\Gamma(k+\frac{\tau}{S})\Gamma(k+\frac{\tau}{S}+2\sigma)}
				= \frac{1}{\sigma^2}
		 	\big(\frac{\Gamma(n+\frac{\tau}{S}+\sigma)^2}{\Gamma(n+\frac{\tau}{S})\Gamma(n+\frac{\tau}{S}+2\sigma)}-\frac{\Gamma(\frac{\tau}{S}+\sigma)^2}{\Gamma(\frac{\tau}{S})\Gamma(\frac{\tau}{S}+2\sigma)}\big).
	\end{equation*}
	Consequently, we have
	\begin{equation}
	\label{ESP-M2}
			 \E[M^2] = \frac{\sigma^2\lambda^2\Gamma(\frac{\tau}{S})}{\Gamma(\frac{\tau}{S}+2\sigma)}
			\Big(\frac{bc}{2\sigma-1}\frac{\tau}{S} 
			+
			(b-c)\frac{b\alpha  - c\beta}{\sigma S}
			+ \frac{(b\alpha  - c\beta)^2}{\sigma^2S^2} \Big) -
			\frac{(b\alpha  - c\beta)^2}{S^2}
	\end{equation}
	and we achieve the proof of Theorem \ref{T-general-as-large} via \eqref{ESP-MW-2} and \eqref{ESP-M2}.
\end{proof}

\section{Proofs of the asymptotic normality results}

\subsection{Traditional urn model}

\begin{proof}{Proof}{\ref{T-tradi-dis}}
We shall make use of part $(b)$ of Theorem 1 and Corollaries 1 and 2 from \cite{Heyde77}.
Let \begin{equation*}
	s_n^2=\sum_{k=n}^\infty \E[\Delta M_k^2].
\end{equation*}
It is not hard to see that 
\begin{equation*}
	\lim_{n\to\infty} s_n^2 =0 
\end{equation*}
since \begin{equation*}
	\sum_{n=1}^\infty \E[\Delta M_n^2] \leq \frac{S^2}{4}\sum_{n=1}^\infty \frac{1}{\tau_n^2}<+\infty.
\end{equation*}
Moreover, using the convergence of $(M_n)$ in $\IL^2$ and the moments of a beta distribution with parameters $\frac{\alpha}{S}$ and $\frac{\beta}{S}$, we get that
\begin{equation*}
	\lim_{n\to \infty}\Big(\sum_{k=n}^\infty \frac{1}{\tau_{k+1}^2}\Big)^{-1}s_n^2= \frac{\alpha\beta S^2}{(\alpha+\beta)(\alpha+\beta+S)},
\end{equation*}
leading to
\begin{equation*}
	\lim_{n\to \infty} n s_n^2= \ell \hspace{1cm}\text{where}\hspace{1cm} \ell=\displaystyle\frac{\alpha\beta}{(\alpha+\beta)(\alpha+\beta+S)}.
\end{equation*}
Hence 
\begin{eqnarray*}
\lim_{n\to\infty} \frac{1}{s_n^2}\sum_{k=n}^\infty \E\bigl[\Delta M_{k+1}^2 | \F_{k} \bigr] & = & \lim_{n\to\infty} \frac{1}{s_n^2}\sum_{k=n}^\infty \frac{c^2M_k(1-M_k)}{\tau_{k+1}^2} \hspace{1cm} \text{a.s.}\\
		& = &   \lim_{n\to\infty}  \frac{1}{\ell S^2}\Big(\sum_{k=n}^\infty \frac{1}{\tau_{k+1}^2}\Big)^{-1} \sum_{k=n}^\infty \frac{S^2M_k(1-M_k)}{\tau_{k+1}^2} \hspace{1cm} \text{a.s.}\\
		& = & \frac{M_\infty(1-M_\infty)}{\ell} \hspace{1cm} \text{a.s.}
\end{eqnarray*}
Consequently, the first condition of part (b) of Corollary 1 in \cite{Heyde77} is satisfied with $\displaystyle\eta^2 = \ell^{-1} M_\infty(1-M_\infty)$. Let us now focus on the second condition of Corollary 1 in \cite{Heyde77} and let $\eps>0$. On the one hand, we get that for all $\eps>0$
\begin{equation*}
	\frac{1}{s_n^2}\sum_{k=n}^\infty \E\bigl[\Delta M_{k+1}^2 \1_{|\Delta M_{k+1}| > \eps s_n} \bigr] 
	 \leq  \frac{1}{\eps^2 s_n^4}\sum_{k=n}^\infty \E\bigl[\Delta M_{k+1}^4  \bigr]
	 \leq  \frac{7S^4}{\eps^2 s_n^4}\sum_{k=n}^\infty \frac{1}{\tau_k^4} 
	 \leq \frac{7}{\eps^2 s_n^4}\sum_{k=n}^\infty \frac{1}{k^4}.
\end{equation*}
On the other and, using that $s_n^4$ increases at speed $n^2$ and that
\begin{equation*}
	\lim_{n\to\infty} 3n^3\sum_{k=n}^\infty \frac{1}{k^4}=1,
\end{equation*}
we can conclude that 
\begin{equation*}
	\displaystyle \lim_{n\to\infty} \frac{1}{s_n^2}\sum_{k=n}^\infty \E\bigl[\Delta M_k^2 \1_{|\Delta M_k| > \eps s_n} \bigr] = 0 \hspace{1cm}\text{a.s.}
\end{equation*}
Hereafter, we easily get that 
\begin{equation}
\label{bracketMn}
	\sum_{k=1}^\infty \frac{1}{s_k^4}\E\big[\Delta M_k^4 | \F_{k-1}\big] \leq 7 \sum_{k=1}^\infty \frac{1}{k^2} <+\infty.
\end{equation} 
Noting that 
\begin{equation*}
	\sum_{k=1}^n\frac{1}{s_k^2}\big(|\Delta M_k|^2 - \E\big[|\Delta M_k|^2 |\F_{k-1}\big]\big)
\end{equation*} 
is a martingale, the equation \eqref{bracketMn} proves that its bracket is convergent, wich implies that the martingale is also convergent. This gives us
\begin{equation*}
	\sum_{k=1}^\infty\frac{1}{s_k^2}\big(|\Delta M_k|^2 - \E\big[|\Delta M_k|^2 |\F_{k-1}\big]\big) < + \infty \hspace{1cm}\text{a.s.}
\end{equation*}
Hence, the second condition of Corollary 1 in \cite{Heyde77} is satisfied. Therefore we obtain that 
\begin{equation}	
  \frac{ M_\infty - M_n}{\sqrt{\langle M\rangle_\infty - \langle M\rangle_n}} \underset{n\to\infty}{\overset{\cL}{\longrightarrow}} \cN\big(0,1\big).
\end{equation}
Moreover, since 
\begin{equation*}
	\lim_{n\to\infty}\sqrt{\frac{M_n(1-M_n)}{n(\langle M\rangle_\infty - \langle M\rangle_n)}} = 1\hspace{1cm}\text{a.s.}
\end{equation*}
we finally obtain from Slutky's Lemma that
\begin{equation}	
  \sqrt{n}\frac{ M_\infty - M_n}{\sqrt{M_n(1-M_n)}} \underset{n\to\infty}{\overset{\cL}{\longrightarrow}} \cN\big(0,1\big).
\end{equation}
which achieves the proof of Theorem \ref{T-tradi-dis}.
\end{proof}

\subsection{Generalized urn model -- small urns}

\begin{proof}{Theorem}{\ref{T-general-dis-small}}
We shall make use of the central limit theorem for multivariate martingales given e.g. by Corollary 2.1.10 in \cite{Duflo97}. First of all, we already saw from \eqref{Mnwnrates} that
	\begin{equation*}
	\lim_{n\to\infty} \frac{1}{w_n}\langle M\rangle_n = (1-2\sigma)\Gamma \hspace{1cm}\text{a.s.}
\end{equation*}
It only remains to show that Linderberg's condition is satisfied, that is for all $\eps >0$,
\begin{equation*}
	\frac{1}{w_n} \sum_{k=0}^{n-1} \E\big[\|\Delta M_{k+1}\|^2 \1_{\|\Delta M_{k+1}\|\geq \eps \sqrt{w_n}}| \F_k\big]
\overset{\IP}{\underset{n\to\infty}{\longrightarrow}} 0.
\end{equation*}
We clearly have
\begin{equation*}
	\frac{1}{w_n} \sum_{k=0}^{n-1} \E\big[\|\Delta M_{k+1}\|^2 \1_{\|\Delta M_{k+1}\|\geq \eps \sqrt{w_n}}| \F_k\big] 
	  \leq \frac{1}{\eps w_n^2} \sum_{k=0}^{n-1} \E\big[\|\Delta M_{k+1}\|^4\big] 
	  \leq\frac{m^2}{\eps w_n^2} \sum_{k=0}^{n-1} \sigma_k^4 \hspace{1cm}\text{a.s.}
\end{equation*}
However, it is not hard to see that
\begin{equation*}
	\lim_{n\to\infty} \frac{1}{w_n^2}\sum_{k=0}^{n-1} \sigma_k^4 = 0
\end{equation*}
which ensures Lindeberg's condition is satisfied. Consequently, we can conclude that 
\begin{equation*}
 	\frac{M_n}{\sqrt{w_n}} \overset{\cL}{\underset{n\to\infty}{\longrightarrow}} 
	\cN \big(0,(1-2\sigma)\Gamma\big). 
 \end{equation*}
As $M_n =\sigma_n\big(U_n - \E[U_n]\big)$ and $\sqrt{n}\sigma_n$ is equivalent to $\sqrt{(1-2\sigma)w_n}$, 
together with \eqref{ESP-sigma}, we obtain that 
\begin{equation*}
	\frac{U_n - n v_1}{\sqrt{n}} \overset{\cL}{\underset{n\to\infty}{\longrightarrow}} 
	\cN \big(0,\Gamma\big).
\end{equation*}
\end{proof}

\subsection{Generalized urn model -- critically small urns}

\begin{proof}{Theorem}{\ref{T-general-dis-crit}}
We shall also make use of the central limit thoerem for multivariate martingales. We already saw from \eqref{Mnwnrates-crit} that
	\begin{equation*}
	\lim_{n\to\infty} \frac{1}{w_n}\langle M\rangle_n 
		= bc\begin{pmatrix}1 & -1 \\ -1 & 1\end{pmatrix}.
\end{equation*}

Once again, it only remains to show that Linderberg's condition is satisfied, that is for all $\eps >0$,
\begin{equation*}
	\frac{1}{w_n} \sum_{k=0}^{n-1} \E\big[\|\Delta M_{k+1}\|^2 \1_{\|\Delta M_{k+1}\|\geq \eps \sqrt{w_n}}| \F_k\big]
\overset{\IP}{\underset{n\to\infty}{\longrightarrow}} 0.
\end{equation*}
As in the proof of Theorem \eqref{T-general-dis-small}, we have
\begin{equation*}
	\frac{1}{w_n} \sum_{k=0}^{n-1} \E\big[\|\Delta M_{k+1}\|^2 \1_{\|\Delta M_{k+1}\|\geq \eps \sqrt{w_n}}| \F_k\big] 
	  \leq \frac{1}{\eps w_n^2} \sum_{k=0}^{n-1} \E\big[\|\Delta M_{k+1}\|^4\big] 
	  \leq \frac{m^2}{2\eps w_n^2} \sum_{k=0}^{n-1} \sigma_k^4. \hspace{1cm}\text{a.s.}
\end{equation*}
It is not hard to see that once again
\begin{equation*}
	\lim_{n\to\infty} \frac{1}{w_n^2}\sum_{k=0}^{n-1} \sigma_k^4=0.
\end{equation*}
Hence, Lindeberg's condition is satisfied and we find that 
\begin{equation*}
 	 \frac{M_n}{\sqrt{w_n}} \overset{\cL}{\underset{n\to\infty}{\longrightarrow}} 
	\cN \big(0,\Gamma\big). 
 \end{equation*}
As $M_n =\sigma_n\big(U_n - \E[U_n]\big)$ and $\sigma_n\sqrt{n\log n}$ is equivalent to $\sqrt{w_n}$, together with \eqref{ESP-sigma}, we can conclude that 
\begin{equation*}
	\frac{U_n - n v_1}{\sqrt{n}} \overset{\cL}{\underset{n\to\infty}{\longrightarrow}} 
	\cN \big(0,\Gamma\big).
\end{equation*}
\end{proof}

\nocite{*}
\bibliographystyle{acm}
\small\bibliography{Biblio-polya.bib}

\end{document}